\documentclass[leqno,12pt]{amsart}
\usepackage{amssymb,amsthm,amsmath,txfonts,a4wide}
\usepackage{xcolor}
\usepackage{enumerate}
\usepackage{hyperref}
\usepackage{tikz}
\usetikzlibrary{shapes,arrows}
\usepackage{times}
\usepackage{verbatim}

\hypersetup{
	colorlinks,
	citecolor=green,
	filecolor=blue,
	linkcolor=red,
	urlcolor=blue
}

\newtheorem{defn}{Definition}[section]

\newtheorem{remark}[defn]{Remark}

\newtheorem{theorem}[defn]{Theorem}
\newtheorem{corollary}[defn]{Corollary}

\newtheorem{definition}[defn]{Definition}
\newtheorem{proposition}[defn]{Proposition}

\newtheorem{claim}[defn]{Claim}

\numberwithin{equation}{section}

\begin{document}
	
	\title{Spectrum of linearized operator at ground states of a system of Klein-Gordon equations}
	\author{Yan Cui}
	\thanks{Yan Cui is supported by the Natural Science Foundation of China (No. 12001555). }
	\address[Yan Cui]{Department of Mathematics, Jinan University, Guangzhou, P. R. China.}
	\email{cuiy32@jnu.edu.cn}

	\author{Bo Xia}
	\address[Bo Xia]{School of Mathematical Sciences,USTC, Hefei, P. R. China.}
	\email{xiabomath@ustc.edu.cn,xaboustc@hotmail.com}
	\thanks{Bo Xia was supported by NSFC No. 12171446.}
	
\author[Kai Yang]{Kai~Yang}
\thanks{Kai Yang was supported by the Jiangsu Shuang Chuang Doctoral Plan and the Natural
Science Foundation of Jiangsu Province(China): BK20200346.}
\address[Kai Yang]
{School of Mathematics, Southeast University, Nanjing, Jiangsu Province, 211189, China}
\email{yangkai99sk@gmail.com, kaiyang@seu.edu.cn}
	\maketitle
	

\begin{abstract}{
	Previously, the existence of ground state solutions of a family of systems of Klein-Gordon equations has been widely studied. In this article, we will study the linearized operator at the ground sate and  give a complete description of the spectrum for this operator in the radial case: the existence of a unique negative eigenvalue, no resonance at `1'(the bottom of the essential spectrum), no embedded eigenvalue in the essential spectrum and the spectral gap property (\textit{i.e.}, there is no eigenvalue in the interval $(0,1]$).}
\end{abstract}

\section{Introduction}	
	In this paper, we are going to study the spectral properties of the operator
		\begin{equation}\label{intr:eq:03230}
		\mathcal{D} := \begin{bmatrix}
			&-\Delta + 1-3Q^2  & 0\\
			& 0 & -\Delta +1 -\frac{3-\beta}{1+\beta}Q^2
		\end{bmatrix}
	\end{equation}
 	over $L^{2}(\mathbb{R}^3)\times L^2(\mathbb{R}^3)$, where $\beta>0$ is a parameter. Here $Q$ is the unique positive ground state \cite{Coffman1972,KMK89} of the equation
 	\begin{equation}\label{eq:kwong}
 		-\Delta\psi(x) +\psi(x) - \psi(x)^3=0,\ \ x\in\mathbb{R}^3.
 	\end{equation}

 	The operator $\mathcal{D}$ arises in the study on stability of the (non-standard) stationary solutions to the system of Klein-Gordon equations, which can be used to describe the motion of charged mesons in
electromagnetic field ({see Segal \cite{segal1965}, or the introduction part of \cite{sirakov07}}). Such solutions obey the following system of elliptic equations
 	\begin{equation}\label{intr:eq:1}
 		\left\{
 		\begin{split}
 			-\Delta u_1 +u_1 -u_1^3 -\beta u_1u_2^2&=0\\
 			-\Delta u_2 +u_2 -u_2^3 -\beta u_1^2u_2&=0
 		\end{split}
 		\right.
 	\end{equation}
 	The existence of solutions to this system of equations was proved in \cite{sirakov07} (see Theorem \ref{thm:1} below). Denote one of such solutions by the couple $(Q_1,Q_2)$. Linearizing \eqref{intr:eq:03230} at $(Q_1,Q_2)$ gives rise to the operator $\mathcal{D}$, see  Section \ref{sec:bground} for the details.
 	
 	In the following, we will denote
 	\begin{equation}
 		L_\beta := -\Delta +1-\frac{3-\beta}{1+\beta}Q^2.
 	\end{equation}
 	Thus we can rewrite
 	\begin{equation}
 		\mathcal{D}= \begin{bmatrix}
 			L_0&0\\
 			0&L_\beta
 		\end{bmatrix}.
 	\end{equation}
 	Observe that for $\beta=0$, $L_0=-\Delta + 1-3Q^2$ and for $\beta=1$, $L_1=-\Delta +1-Q^2$. Recall that $L_0$ and $L_1$ are just these operators $L_{\pm}$, which have been widely studied in {\cite{Weinstein1982,nakanishi2007,costinhuangschlag2014,Demanet2006}}. For reader's convenience, we summarize the results about the spectrum of these two operators as follows {(see Section \ref{sec:pre} for notations and notions)}.
	\begin{theorem}\label{thm:known:1}
		Let $L_0$ and $L_1$ be as above, the following assertions hold.
		\begin{enumerate}[(i)]
			\item Both $L_0$ and $L_1$ are self-adjoint operators over $L^2(\mathbb{R}^3)$.
			\item $L_0$ has a unique negative eigenvalue, denoted by $\lambda_0$, while $L_1$ is a positive operator and hence it has no negative spectrum.
			\item The point $0$ is an eigenvalue of $L_0$ with multiplicity three over $L^2(\mathbb{R}^3)$, and the corresponding eigenspace is spanned by $\partial_{x_1}Q,\partial_{x_2}Q$ and $\partial_{x_3}Q$.
			\item The point $0$ is an eigenvalue of $L_1$ with multiplicity one over $L^2(\mathbb{R}^3)$, and the corresponding eigenspace is spanned by $Q$.
			\item {For both operators $L_0$ and $L_1$, the interval $(0,1]$ contains no spectral points}. What's more, the {threshold} $1$ is not a resonance for both of them {(in the radial direction)}.
			\item Both of these two operators have continuous spectrum $[1,\infty)$, which contains no embedded eigenvalues.
		\end{enumerate}
	\end{theorem} 	
	\begin{proof}
		 The proof of $(i)$ can be found in \cite{helffer2013,Reed1978} and that of $(ii)$-$(iv)$ in \cite{Weinstein1985}. For the proof of $(vi)$, one can invoke results of Kato \cite{Kato1959}.
		
		The proof of $(v)$ has a little bit long history: Demanet and Schlag first numerically verified this assertion in \cite{Demanet2006}; by first establishing the piece-wise analytic approximation of the ground sate, Costin, Huang and Schlag proved it under the radial assumption in \cite{costinhuangschlag2014}; recently, using this approximation of $Q$ and Sturm's comparison theorem, Li and the third author of the present paper finally proved this result in the nonraidal setting in \cite{LY}.
	\end{proof}
	\begin{remark}
		 To characterize the spectrum of $L_0, L_1$ and even $L_\beta$ for $\beta\in(0,1)$ completely, we shall mention the notion `singular continuous part of the spectrum' (see \cite{Reed1978} for the definition). Since the ground sate $Q$ is spherically symmetric and decays exponentially in the radial direction, we can apply \cite[Theorem XIII.21]{Reed1978} to see that the singular continuous part of the spectrum of $L_\beta$ is an empty set for each $\beta\in[0,1]$.
	\end{remark}
	
  In particular, we know from Theorem \ref{thm:known:1} the full spectrum of $L_0$. Note that $L_0$ and $L_\beta$ are the only non-zero entries of $\mathcal{D}$. Thus, in order to study the spectral properties of $\mathcal{D}$, we first work out the spectrum of $L_\beta$.
	
	\begin{theorem}\label{thm:main} For each fixed number $\beta\in(0,1)$, the following assertions hold.
		\begin{enumerate}[(i)]
			\item The operator $L_\beta$ has only one negative eigenvalue, say $\lambda_\beta$, over $L^2(\mathbb{R}^3)$. What's more, the corresponding eigenfunction is a smooth radial function that does not change sign.
			\item The point $0$ is not in the discrete spectrum of $L_\beta$ over $L^2(\mathbb{R}^3)$.
			\item The operator $L_\beta$ has no discrete spectrum in the open interval $(0,1)$ over $L^2_{\mathrm{rad}}(\mathbb{R}^3)$.
			\item The point $1$ is not an eigenvalue or a resonance of $L_\beta$ over $L^2_{rad}(\mathbb{R}^3)$.
			\item The interval $[1,\infty)$ is the essential spectrum of $L_\beta$ over $L^2(\mathbb{R}^3)$, which does not contain any embedded eigenvalue.
		\end{enumerate}	
	\end{theorem}	

	\begin{remark}
		Indeed, over the space of square integrable radial functions, any real number is not a resonance:
		Theorem C.4.2 in \cite{Simon1982} implies that any real number $\lambda <1$ is not a resonance for $L_\beta$; it follows from \cite[Section 3]{Agmon1975} that any real number $\lambda>1$ is a resonance.
	\end{remark}

	An immediate corollary is
	\begin{corollary}\label{cor:main}
		Fix $\beta\in(0,1)$. Over the product space $L^2_{\mathrm{rad}}(\mathbb{R}^3)\times L^2_{\mathrm{rad}}(\mathbb{R}^3)$, the spectrum of $\mathcal{D}$ is given by
		\begin{equation*}
			\{\lambda_0\}\cup \{\lambda_\beta\}\cup [1,\infty).
		\end{equation*}
		What's more, $\lambda_0$ and $\lambda_\beta$ are simple eigenvalues, the point $1$ is not an eigenvalue or a resonance, and the set $[1,\infty)$ is the essential spectrum with no embedded eigenvalues.
	\end{corollary}
	\begin{proof}
		We first determine the eigenvalues of $\mathcal{D}$. Let $\psi_0$ be the eigenfunction of $L_0$ corresponding to $\lambda_0$ and $\psi_\beta$ the eigenfunction of $L_\beta$ corresponding to $\lambda_{\beta}$. Then one can check that $\lambda_0$ and $\lambda_{\beta}$ are eigenvalues of $\mathcal{D}$ with their corresponding eigenfunctions $(\psi_0,0)^{\mathsf{T}}$ and $(0,\psi_\beta)^{\mathsf{T}}$ respectively. We shall show they are the only eigenvalues of $\mathcal{D}$. Let $\lambda$ be an eigenvalue of $\mathcal{D}$, then we take a non-zero eigenfunction $h=(h_1,h_2)^{\mathsf{T}}\in H^2_{rad}(\mathbb{R}^3)\times H^2_{rad}(\mathbb{R}^3)$ so that
		\begin{equation}
			L_0h_1=\lambda h_1,\ \ \mathrm{and}\ \ L_{\beta}h_2=\lambda h_2.
		\end{equation}
		If $h_1$ is not a zero function, then it follows that $\lambda$ is an eigenvalue of $L_0$. Recalling that the situation here is in the radial setting, we see that $\lambda_0$ is the unique eigenvalue of $L_0$ and hence $\lambda=\lambda_0$. If $h_2$ is a non-zero function, we can use the same argument to conclude $\lambda=\lambda_{\beta}$. In particular, the point $1$ is not an eigenvalue.
		
		It follows from Theorems \ref{thm:known:1} and \ref{thm:main} that $\sigma_{ess}(L_0)=\sigma_{ess}(L_{\beta})=[1,\infty)$. Then we conclude from Proposition \ref{prop:abs:2} that $\sigma_{ess}(\mathcal{D})=[1,\infty)$. What's more, since $\lambda_0$ and $\lambda_{\beta}$ are all the eigenvalues of $\mathcal{D}$ and they are both smaller than the bottom edge $1$, the operator $\mathcal{D}$ does not have any embedded eigenvalues in $\sigma_{ess}(\mathcal{D})$.
		
		We know from Theorems \ref{thm:1} and \ref{thm:main} that the threshold point `1' is not a resonance for both $L_0$ and $L_\beta$. We then infer directly from Definition \ref{def:1} that it is not a resonance for $\mathcal{D}$.
	\end{proof}

	\begin{remark}
	As is remarked by Marzuola and Simpson (see \cite{Marzuola2010}) in the case of Schr\"{o}dinger equation, the existence of both bottom resonances and embedded eigenvalues requires `specific nature' of the potentials and they are believed to be rare. We indeed here verified this philosophy for the system of coupled Schr\"{o}dinger equations, even though the corresponding ground state solutions and hence the potentials enjoy very nice properties like spherically symmetry, exponential decaying and smoothness.
	\end{remark}
	
	\begin{remark}
		Specifically, Corollary \ref{cor:main} implies the spectral instability of the ground sate of the coupled Klein-Gordon systems
 	\begin{equation}\label{intr:eq:damping}
 		\left\{
 		\begin{split}
 			\partial_t^2 u_1-\Delta u_1 +u_1 -u_1^3 -\beta u_1u_2^2&=0\\
 			\partial_t^2 u_2-\Delta u_2 +u_2 -u_2^3 -\beta u_1^2u_2&=0
 		\end{split}
 		\right.\ \ \ \ \ \beta>0
 	\end{equation}
 provided that the ground state $(Q_1,Q_2)\in H^1\times H^1$ is unique (see Remark \ref{remark:1}). However, as in the case of Schr\"{o}dinger equation considered in \cite{Weinstein1985}, we expect that this ground state be orbitally stable.
 It is worth mentioning that, according to Theorem \ref{thm:main}, one might mimic the general theory of orbital stability and instability of abstract Hamiltonian systems proposed by Grillakis, Shatah and Stauss \cite{GSS1987} to study this issue.
	\end{remark}

	This article is preceding as follows: in Section \ref{sec:pre}, we will introduce functional notations and recall elementary notions in spectral theory; in Section \ref{sec:bground}, we will derive our matrix operator $\mathcal{D}$ by a linearization process; in the last Section \ref{sec:proof}, we will prove our main Theorem \ref{thm:main}.

\section{Preliminaries}\label{sec:pre}
	In this section, we give the symbolic notations in this article and recall notions concerning spectrum of self-adjoint operators.
	
	We will use $L^2$ to denote the space of square integrable real-valued functions defined on the whole Euclidean space $\mathbb{R}^3$ and $L^2_{rad}$ the space consisting of radial functions in $L^2$. We also use $H^1$ to denote the subspace of $L^2$, each element of which has square integrable gradients. Similarly, we also use $H^1_{rad}$ to denote the subspace of $H^1$, whose elements are spherically symmetric.

	Since the functions we consider here are real valued, we restrict ourselves to real Hilbert spaces. We first recall
	\begin{definition}
		Let $\left(\mathsf{H},\langle\cdot,\cdot\rangle\right)$ be a real Hilbert space, and $\left(\mathrm{Dom}(A),A\right)$ be densely defined linear operator from $\mathrm{Dom}(A)$ to $\mathsf{H}$.
		\begin{enumerate}
			\item Denote $\mathrm{Dom}(A^\ast):=\left\{u\in \mathsf{H}: \exists u^\ast\in \mathsf{H}\ \mathrm{s.t.}\ \langle u,Av\rangle = \langle u^\ast,v\rangle,\forall v\in\mathrm{Dom}(A)\right\}$ and set $A^\ast u:=u^\ast$ for all $u\in\mathrm{Dom}(A^\ast)$. Note that for each $u\in \mathrm{Dom}(A^\ast)$, $u^\ast$ is uniquely determined and hence $A^\ast$ is well defined. We say $(\mathrm{Dom}(A^\ast),A^\ast)$ the adjoint of $(\mathrm{Dom}(A),A)$.
			\item  $\left(\mathrm{Dom}(A),A\right)$ is said to be self-adjoint if $\left(\mathrm{Dom}(A),A\right)=\left(\mathrm{Dom}(A^\ast),A^\ast\right)$
		\end{enumerate}
 	\end{definition}
	Usually, one can use {Kato-Rellich theorem (\cite{helffer2013})} to see if the perturbation of the Laplace operator $-\Delta$ by adding a potential is self-adjoint or not.
 	{In the following, we will denote $(\mathrm{Dom}(A),A)$ simply as $A$, if there is no need to emphasize its defining domain.} We next recall
 	\begin{definition}
 		Let $\left(\mathrm{Dom}(A),A\right)$ be a self-adjoint operator over some real Hilbert space $\left(\mathsf{H},\langle\cdot,\cdot\rangle\right)$.
 		\begin{enumerate}
 			\item The subset $\rho(A):=\left\{\lambda\in\mathbb{C}:A-\lambda I:\mathrm{Dom}(A)\rightarrow\mathsf{H}\ \mathrm{is\ bijective}\right\}$ of the complex plane $\mathbb{C}$ is called the resolvent set of $A$.
 			\item  The set $\sigma(A):=\mathbb{C}\backslash\rho(A)$ is called the spectrum of $A$.
 			\item A number $\lambda$ is an eigenvalue of $A$ if the equation $(A-\lambda I)u=0$ admits a non-zero solution in $\mathsf{H}$. The point spectrum $\sigma_{pt}(A)$ of $A$ consists of eigenvalues of $A$.
 			\item The subset $\sigma_c(A):=\left\{\lambda\in\mathbb{C}:A-\lambda I \mathrm{\ is\ one\ to\ one\ but\ not\ onto}\right\}$\footnote{ The set of residual spectrum of self-adjoint operator $A$ is empty. }  is called the continuous spectrum of $A$.
 		\end{enumerate}
 	\end{definition}
 	
	In some case, an eigenvalue can be embedded into the continuous spectrum. In order to distinguish these embedded eigenvalues with others, we introduce the notion of discrete spectrum.
	\begin{definition}
		Let $A$ be a self-adjoint operator on a real Hilbert space $\mathsf{H}$.
		\begin{enumerate}
			\item The discrete spectrum $\sigma_{disc}(A)$ of $A$ is the set, consisting of isolated points of the spectrum which correspond to eigenvalues, the eigenspace of which is of finite dimension.
				
			\item The essential spectrum $\sigma_{ess}(A)$ of $A$ is defined to be $\sigma(A)\backslash\sigma_{disc}(A)$.
		\end{enumerate}
	\end{definition}	

	To {see} whether or not one point belongs to essential spectrum, we will use
	\begin{theorem}[Weyl's criterion, \cite{helffer2013}]\label{thm:weyl}
		Let $T$ be a self-adjoint operator. Then $\lambda$ belongs to the essential spectrum if and only if there exists a sequence $u_n$ in $\mathrm{Dom}(T)$ with $\|u_n\|=1$, $u_n\rightharpoonup_{n\rightarrow\infty}0$ and $\|(T-\lambda)u_n\|\rightarrow_{n\rightarrow\infty}0$.
	\end{theorem}
	Using this result, we will determine the essential spectrum of a matrix operator of the diagonal form, from that of its diagonal entities.
\subsection{Abstract Theory}	\label{sec:abs:theory}
	Let $(\mathrm{Dom}(T_1),T_1)$ and $\left(\mathrm{Dom}(T_2),T_2\right)$ be two densely defined operators on the Hilbert space $\mathsf{H}$. We then consider the operator of matrix form
	\begin{equation}
		\mathcal{T}:=\begin{bmatrix}
			T_1& 0\\
			0& T_2
		\end{bmatrix}
	\end{equation}
	with domain $\mathrm{Dom}(T_1)\times \mathrm{Dom}(T_2)$ on $\mathsf{H}\times\mathsf{H}$. We would like to withdraw the information about spectrum of $\left(\mathrm{Dom}(T_1)\times\mathrm{Dom}(T_2),\mathcal{T}\right)$ from those about $T_1$ and $T_2$.

	{Considering the form of our operator $\mathcal{D}$}, we will assume $(\mathrm{Dom}(T_1),T_1)$ and $(\mathrm{Dom}(T_2),T_2)$ are both self-adjoint on $\mathsf{H}$. Then the operator $\left(\mathrm{Dom}(T_1)\times\mathrm{Dom}(T_2),\mathcal{T}\right)$ is also a self-adjoint operator on $\mathsf{H}\times\mathsf{H}$ and hence $\sigma(\mathcal{T})\subset\mathbb{R}$. What's more, its essential spectrum is completely determined by those of $T_1$ and $T_2$.
\begin{proposition}\label{prop:abs:2} Under the above assumptions on $T_1$ and $T_2$, we have
	$$\sigma_{ess}(T_1)\cup \sigma_{ess}(T_2)=\sigma_{ess}(\mathcal{T})$$
\end{proposition}
\begin{proof}
	Let $\lambda\in \sigma_{ess}(T_1)\cup \sigma_{ess}(T_2)$. For convenience, we may assume $\lambda\in\sigma_{ess}(T_1)$. We shall use a Weyl sequence of $T_1$ corresponding to $\lambda$ to construct one for {$\mathcal{T}$ and the same $\lambda$}. The key observation is
	\begin{equation}\label{eq:03210}
		T_1h=\lambda h\ \Leftrightarrow \begin{bmatrix}
			T_1& 0\\
			0& T_2
		\end{bmatrix} \begin{bmatrix} h\\ 0\end{bmatrix} =\mathcal{T} \begin{bmatrix} h\\ 0\end{bmatrix} =\lambda \begin{bmatrix} h\\ 0\end{bmatrix}.
	\end{equation}
	with $\left\|(h,0)^{\mathsf{T}}\right\|_{\mathsf{H}\times\mathsf{H}}=\left\|h\right\|_{\mathsf{H}}$.
	By {Weyl's criterion (Theorem \ref{thm:weyl})}, we take $h_n\in\mathrm{Dom}(T_1)$ to be a Weyl sequence corresponding to $(T_1,\lambda)$. Then by \eqref{eq:03210}, $(h_n,0)^\mathsf{T}\in\mathrm{Dom}(\mathcal{T})$ is a Weyl sequence corresponding to $(\mathcal{T},\lambda)$. This implies $\lambda\in\sigma_{ess}(\mathcal{T})$, with the help of Theorem \ref{thm:weyl}.

	{For the inverse inclusion, we take $\lambda\in\sigma_{ess}(\mathcal{T})$. By Theorem \ref{thm:weyl}, we pick a Weyl's sequence $(h_n,f_n)^{\mathsf{T}}\in\mathrm{Dom}(\mathcal{T})$ corresponding to $(\mathcal{T},\lambda)$, satisfying
	\begin{equation}
		\left\|(h_n,f_n)^{\mathsf{T}}\right\|_{\mathsf{H}\times\mathsf{H}}=1,\ (h_n,f_n)^{\mathsf{T}}\rightharpoonup_{n\rightarrow\infty}(0,0)^{\mathsf{T}},\ \textrm{and}\ \left\|(\mathcal{T}-\lambda)(h_n,f_n)^{\mathsf{T}}\right\|_{\mathsf{H}\times\mathsf{H}}\rightarrow_{n\rightarrow\infty}0.
	\end{equation}	
	Using Pigeon's hole principle, we can take a subsequence $(h_{n_k},f_{n_k})^{\mathsf{T}}\in\mathrm{Dom}(\mathcal{T})$, one component (say for instance the first component) of which satisfies	
	\begin{equation}
		\left\|h_{n_k}\right\|_{\mathsf{H}}\geq \frac{1}{2},\forall k\geq 1; h_{n_k}\rightharpoonup_{k\rightarrow\infty}0,\ \mathrm{and}\ \left\|(T_1-\lambda)h_{n_k}\right\|_{\mathsf{H}}\rightarrow_{k\rightarrow\infty}0.
	\end{equation}	
	By a renormalization argument, we can obtain a Weyl's sequence for $(T_1,\lambda)$. It then follows from Theorem \ref{thm:weyl} that $\lambda\in\sigma_{ess}(T_1)$, which completes the proof of Proposition \ref{prop:abs:2}.}
\end{proof}

	At the end of this section, we recall the notion of resonances. We will use the definition for scalar equation as in \cite{kenigmendelson,costinhuangschlag2014}. We also {adopt} the notion of resonances {as in \cite{Schlag},}  since our operator $\mathcal{D}$ is of the matrix form.
	\begin{definition}\label{def:1}
		Let $\lambda$ be a real number. We say that $\mathcal{D}$ has a resonance at $\lambda$, by meaning that there exists a distributional solution $f=(f_1,f_2)^{\mathsf{T}}\notin L^{2}\times L^2$ to the equation
		\begin{equation}
			\mathcal{D}f=\lambda f
		\end{equation}
		that satisfies
		\begin{equation*}
			\int_{\mathbb{R}^3}\left(|f_1(x)|^2+|f_2(x)|^2\right)(1+|x|)^{-2\gamma}dx<\infty,\ \ \forall \gamma>\frac{1}{2}.
		\end{equation*}
		{If $\mathcal{D}$ has a resonance at $\lambda$, we also say simply $\lambda$ is a resonance of $\mathcal{D}$, or a resonance of $\mathcal{D}$ occurs at $\lambda$.}
	\end{definition}

\section{Derivation of the matrix operator $\mathcal{D}$}\label{sec:bground}	
	In this section, we are going to recall how the system \eqref{intr:eq:1} is solved and to derive $\mathcal{D}$.
	
\subsection{Solutions to the system of {Schr\"{o}dinger} equations}
	Recall that the system we are going to consider is
	\begin{equation}\label{eq:1}
		\left\{
		\begin{split}
			-\Delta u_1 +u_1 -u_1^3 -\beta u_1u_2^2&=0\\
			-\Delta u_2 +u_2 -u_2^3 -\beta u_1^2u_2&=0
		\end{split}
		\right.
	\end{equation}
	where $u_i:\mathbb{R}^3\ni x\longmapsto u_i(x)\in\mathbb{R}, i=1,2$ are the unknowns. This system is solved by the variation argument as follows. We first associate to \eqref{eq:1} the energy functional
	\begin{equation}\label{eq:2}
		E[u_1,u_2]:=\frac{1}{2}\int_{\mathbb{R}^3}\left[\sum_{i=1,2}\left(\left|\nabla u_i\right|^2+|u_i|^2\right)\right]dx -\frac{1}{4}\int_{\mathbb{R}^3}\left[u^4_1+u^4_2+2\beta u_1^2u_2^2\right]dx.
	\end{equation}
	
	{Usually, solutions to \eqref{eq:1} that are of much interest is these, each component of which is not identically zero.} In order to obtain such solutions, we introduce the following constraint set
		\begin{equation*}
		\mathcal{N}_1:=\left\{(u_1,u_2)\in \left(H^1\backslash\{0\}\right)^2: \int_{\mathbb{R}^3}\left(\left|\nabla u_i\right|^2+|u_i|^2\right)dx = \int_{\mathbb{R}^3}\left[u^4_i+\beta u_1^2u_2^2\right]dx, i=1,2\ \ \right\}.
	\end{equation*}

	For the following constraint minimizing problems
	\begin{align}\label{minipro}
		A&:=\inf_{(u_1,u_2)\in\mathcal{N}_1}E[u_1,u_2],\\
		A_r&:=\inf_{(u_1,u_2)\in\mathcal{N}_1\cap \left(H^1_{rad}\times H^1_{rad}\right)}E[u_1,u_2],
	\end{align}
	Sirakov proved
	\begin{theorem}[\cite{sirakov07}]\label{thm:1}
		For each $\beta\in{[0,+\infty)}$, one has
		\begin{enumerate}[(i)]
			\item $A=A_r$;
			\item both of these minimum are attained by the couple
			\begin{equation*}
				\left(\sqrt{\frac{1}{1+\beta}} Q,\sqrt{\frac{1}{1+\beta}} Q\right)=:(Q_1,Q_2)
			\end{equation*}
			where $Q(x)$ is the unique positive ground state solution to \eqref{eq:kwong}.
		\end{enumerate}
	\end{theorem}

	{By a variational approach, minimizers for $A$ are weak solutions of System \eqref{eq:1} (see \cite[Lemma 3]{linwei08} or \cite[Remark 3]{sirakov07})}. For other types of solutions and the uniqueness of solutions, we give the following remarks.
	\begin{remark}\label{rem:1215}
		In fact, the system \eqref{eq:1} does possess solutions with one component being identically zero, say for instance $(Q,0)^{\mathsf{T}}$. Theoretically, using the Nehari manifold
		\begin{equation*}
			\mathcal{N}_0:=\left\{(u_1,u_2)\in H^1\times H^1\backslash\{(0,0)\}: \int_{\mathbb{R}^3}\left[\sum_{i=1,2}\left(\left|\nabla u_i\right|^2+|u_i|^2\right)\right]dx = \int_{\mathbb{R}^3}\left[u^4_1+u^4_2+2\beta u_1^2u_2^2\right]dx\right\},
		\end{equation*}
		as the constraint set, for the following minimizing problem
		\begin{equation*}
			A_0:=\inf_{(u_1,u_2)\in\mathcal{N}_0}E[u_1,u_2].
		\end{equation*}
		Sirakov also proved: $A_0$ is a positive number that is assumed by a couple $(u_1,u_2)$ of radial functions in $H^1\times H^1\backslash\{(0,0)\}$.
	\end{remark}

	\begin{remark}\label{remark:1}
		{In \cite{sirakov07}, Sirakov called the extremizers for $A_0$ nontrivial solutions to \eqref{eq:1}, and that for $A$ non-standard ones. In \cite[Remark 2]{sirakov07}, Sirakov conjectured that $(Q_1,Q_2)$ is the unique positive solution for \eqref{eq:1} up to translations.
		{This question has attracted several attentions}: using implicit function theorem, Wei and his collaborators \cite{linwei08,WY12} proved the uniqueness of such solutions in the cases $\beta>1$ and $0<\beta<\beta_0$ for some small $\beta_0$; based on some form of bifurcation method, Chen and Zou \cite{CZ13} proved this conjecture in the case $1>\beta>\beta_1$ for some $\beta_1$ close to $1$. {Whether or not this conjecture is true for} $\beta\in[\beta_0,\beta_1]$ is still open up to now. {What's more,} Chen and Zou \cite{CZ12} proved a weaker version of this conjecture, asserting that $(Q_1,Q_2)$ in Theorem \ref{thm:1} is the unique solution (up to translations) of minimizing problem \eqref{minipro}.}
	\end{remark}
	
\subsection{Derivation of $\mathcal{D}$}	
	Let $(Q_1,Q_2)$ be a nontrivial solution to \eqref{eq:1} (guaranteed by Theorem \ref{thm:1}). By linearizing \eqref{eq:1} at this solution, we arrive at the operator 	
	\begin{equation}
		\mathcal{L}:= \begin{bmatrix}
			& -\Delta +1-\frac{3+\beta}{1+\beta}Q^2  & \frac{-2\beta}{1+\beta}Q^{2}\\
			& \frac{-2\beta}{1+\beta}Q^{2} & -\Delta +1-\frac{3+\beta}{1+\beta}Q^2
		\end{bmatrix}
	\end{equation}
	{		Since $Q$ decays exponentially, one can use {Kato-Rellich theorem} to verify:
		\begin{enumerate}
			\item The domain of $\mathcal{L}$ is $H^2\times H^2=:\mathrm{Dom}(\mathcal{L})$.
			\item The densely defined operator $\left(\mathrm{Dom}(\mathcal{L}),\mathcal{L}\right)$ is a self-adjoint operator acting on $L^2\times L^2$. Consequently the spectrum of $\mathcal{L}$ is contained in the real line $\mathbb{R}$.
		\end{enumerate}
		}

	In order to diagonalize $\mathcal{L}$, we
	introduce an isometric {action} from $L^2\times L^2$ into itself induced by the matrix
	\begin{equation*}
		W:= \frac{1}{\sqrt{2}} \begin{bmatrix}
			& 1 &1\\
			& 1 & -1
		\end{bmatrix}
	\end{equation*}More precisely, this action is {given by}
	\begin{align*}
		W:& L^2(\mathbb{R}^3)\times L^2(\mathbb{R}^3)\rightarrow L^2(\mathbb{R}^3)\times L^2(\mathbb{R}^3)\\
		  & (h_1,h_2)^{\mathsf{T}}\longmapsto \left(\frac{h_1+h_2}{\sqrt{2}},\frac{h_1-h_2}{\sqrt{2}}\right)^{\mathsf{T}}.
	\end{align*}
	 We then diagonalize $\mathcal{L}$ as
	\begin{equation}\label{eq:03230}
		W^{-1}\mathcal{L}W = \begin{bmatrix}
									&-\Delta + 1-3Q^2  & 0\\
									& 0 & -\Delta +1 -\frac{3-\beta}{1+\beta}Q^2
							\end{bmatrix}\equiv\mathcal{D}.
	\end{equation}
	
	{This derivation of $\mathcal{D}$ seems not closely related to the original equation \eqref{eq:1}.} Indeed, this operator can be derived directly by the following linearizing procedure.
	Let $(u_1,u_2)$ be solution to \eqref{eq:1}, and the functions $v_1,v_2$ be defined implicitly by
	\begin{equation}
		(u_1,u_2)^{\mathsf{T}}=W(v_1,v_2)^{\mathsf{T}}.
	\end{equation}
	Then $v_1$ and $v_2$ solve the system
	\begin{equation}\label{eq:0420}
		\left\{
		\begin{split}
			-\Delta v_1+v_1-\frac{1}{4}\left[(v_1+v_2)^3+(v_1-v_2)^3\right]-\frac{\beta}{2}\left(v_1^2-v_2^2\right)v_1&=0\\
			-\Delta v_2+v_2-\frac{1}{4}\left[(v_1+v_2)^3-(v_1-v_2)^3\right]+\frac{\beta}{2}\left(v_1^2-v_2^2\right)v_2&=0
		\end{split}
		\right.
	\end{equation}
	
	Since $(Q_1,Q_2)^{\mathsf{T}}$ is a solution to \eqref{eq:1}, the transformed couple $(\tilde{Q}_1,\tilde{Q}_2)^{\mathsf{T}}:=W(Q_1,Q_2)^{\mathsf{T}} =\left(\sqrt{\frac{2}{1+\beta}}Q,0\right)^{\mathsf{T}}$ is a solution to \eqref{eq:0420}.
	Linearizing the system \eqref{eq:0420} at $(\tilde{Q}_1,\tilde{Q}_2)^{\mathsf{T}}$, gives rise to the linearized operator
	\begin{equation}
		\begin{bmatrix}
			-\Delta+1-3Q^2 & 0\\
			0& -\Delta+1-\frac{3-\beta}{1+\beta}Q^2
		\end{bmatrix}
	\end{equation}
	which is just the diagonal operator $\mathcal{D}$.
	
	\begin{remark}
		As is alluded to in Remark \ref{rem:1215}, the couple $(Q,0)$ is also a solution to \eqref{eq:1}. Linearizing \eqref{eq:1} at this nontrivial but standard solution, gives rise to the following diagonal operator
		\begin{equation}
			\mathcal{L}_1=\begin{bmatrix}
				&-\Delta +1 -3Q^2  &0 \\
				& 0  &-\Delta +1 -\beta Q^2
			\end{bmatrix}.
		\end{equation}	
		Noting that this operator for each $\beta\in(1,3)$ is exactly the operator $\mathcal{D}$ with the corresponding parameter being $\frac{3-\beta}{1+\beta}$, we can {thus} determine the spectrum of $\mathcal{L}_1$ for each such $\beta\in (1,3)$.
	\end{remark}

	\begin{remark}
	Indeed, Sirakov\cite{sirakov07} considered the following more general system
    \begin{equation}\label{eq:gen}
		\left\{
		\begin{split}
			-\Delta u_1 + u_1 -\mu_1u_1^3 -\beta u_1u_2^2&=0,\\
			-\Delta u_2 + u_2 -\mu_2u_2^3 -\beta u_1^2u_2&=0.
		\end{split}
		\right.
	\end{equation}
	and showed that it admits the ground state solution
$(Q_1,Q_2)=(\sqrt{k}Q,\sqrt{l}Q)$ where $k,l$ are determined by the linear system of equations
    \begin{equation}
    \begin{cases}
    k\mu_1+l\beta=1,\\
    k\beta+l\mu_2=1.
    \end{cases}
    \end{equation}
	By following the linarization approach in this section, we may linearize \eqref{eq:gen} at $(Q_1,Q_2)$ to obtain the operator:
	\begin{equation}
	\tilde{\mathcal{L}}=\begin{bmatrix}
	&-\Delta +1 -3 Q^2  & 0 \\
	& 0 &-\Delta +1 -\frac{3(\mu_1-\beta)(\mu_2-\beta)}{\mu_1\mu_2-\beta^2} Q^2
	\end{bmatrix}.
	\end{equation}
	When $\frac{3(\mu_1-\beta)(\mu_2-\beta)}{\mu_1\mu_2-\beta^2}\in (1,3)$, we can apply Theorem \ref{thm:main} to determine the spectrum of $\tilde{\mathcal{L}}$ in the radial case.

	We expect that the spectral problem associated to some even more general systems (see \cite{ac07,CZ12,sirakov07}) than \eqref{eq:gen} can also be considered.
	\end{remark}

\section{Proof of Theorem \ref{thm:main}}\label{sec:proof}
	We devote this section to proving our main result: Theorem \ref{thm:main}. For each $\beta\in(0,1)$, we recall
	\begin{equation}
		L_{\beta}=-\Delta+1-\frac{3-\beta}{1+\beta}Q^2.
	\end{equation}
	We first determine the essential spectrum of $L_\beta$.
	\begin{proof}[Proof of $(v)$] It follows from $\sigma(-\Delta)=\sigma_{ess}(-\Delta)=[0,\infty)$ that $\sigma_{ess}(-\Delta+1)=[1,\infty)$. Since
		the ground state $Q$ decays exponentially at infinity, the multiplication operator induced by the function $\frac{3-\beta}{1+\beta}Q^2$ is relatively compact perturbation of $-\Delta +1$. It then follows from {the Weyl Theorem} (see \cite{helffer2013}) that
		\begin{equation}
			\sigma_{ess}(L_\beta)=\sigma_{ess}(-\Delta +1)=[1,\infty)
		\end{equation}
		
		Note that the ground state solution $Q$ decays faster than the inverse of any given polynomial. This allows us to use the result in \cite{Kato1959} to conclude that $L_\beta$ has no embedded eigenvalue above the lower edge of its continuous spectrum. This finishes the proof of $(v)$ in Theorem \ref{thm:main}.
	\end{proof}

	About the discrete spectrum, we begin our argument by determining the negative ones.
	\begin{proof}[Proof of $(i)$]
		For the existence, using the equation satisfied by $Q$, we do some calculations to obtain
		\begin{align}
			\left\langle L_\beta Q,Q \right\rangle= -\frac{2(1-\beta)}{1+\beta}\left\langle Q^{4},1\right\rangle.
		\end{align}
		Thanks to $\beta<1$ and the fact that $Q$ does not vanish, we have $\left\langle L_\beta Q,Q\right\rangle<0$. This shows that $L_\beta$ must have at least one negative eigenvalue. Thus to obtain the asserted result, it suffices to show that it has at most one negative eigenvalue.
		
		{To achieve this aim, we can indeed mimic the proof for the same statement for $L_0$, see for instance \cite{nakanishischlag2011}}. But here, we present a concise proof, by invoking an application of Glaszmann's lemma (\cite{Birman1987}).
		
		Let $\mu_1\leq\mu_2$ be the first two eigenvalues of $L_\beta$. Noting that the first eigenvalue is a simple one, we can assume even strongly $\mu_1<\mu_2$. {As is shown previously, $\sigma_{ess}(L_\beta)=\sigma_{ess}(L_0)=[1,\infty)$ and $L_\beta\geq L_0$.} This allows us to conclude from \cite[Theorem 4 on Page 227]{Birman1987} with $\gamma=1$ that
		\begin{equation}\label{eq:04147}
			\mu_1\geq \lambda_0,\ \ \mu_2\geq 0.
		\end{equation}
		This result shows that the second eigenvalue (if it exists) of $L_\beta$ is nonnegative, which means there is at most one negative eigenvalue. This finishes the proof of $(i)$ in Theorem \ref{thm:main}.
	\end{proof}
	
	By further exploiting the comparison argument, we can compare these eigenvalues $\lambda_\beta$ {for different values of $\beta$}.
	\begin{proposition}\label{lem:2}
		Fix $0\leq \beta_1<\beta_2\leq 1$. Let  $\lambda_{\beta_1}$ and $ \lambda_{\beta_2}$ be the first eigenvalue of $L_{\beta_1}$ and $L_{\beta_2}$ respectively. Then $\lambda_{\beta_1}<\lambda_{\beta_2}$.
	\end{proposition}
	\begin{proof}
		Suppose otherwise that $\lambda_{\beta_1}\geq\lambda_{\beta_2}$. Take one eigenfunction $f_{\beta_2}$ such that $L_{\beta_2}f_{\beta_2}=\lambda_{\beta_2} f_{\beta_2}$.
		Since {$\lambda_{\beta_1}$} is the first eigenvalue of $L_{\beta_1}$, we have
		$$
		0\leq\left\langle(L_{\beta_1}-\lambda_{\beta_1})f_{\beta_2},f_{\beta_2})\right\rangle.
		$$
		On the other hand, from the assumption  $\lambda_{\beta_1}\geq\lambda_{\beta_2}$, we get
		$$\left\langle(L_{\beta_1}-\lambda_{\beta_1})f_{\beta_2},f_{\beta_2}\right\rangle \leq \left\langle(L_{\beta_1}-\lambda_{\beta_2})f_{\beta_2},f_{\beta_2}\right\rangle =-\frac{3(\beta_2-\beta_1)}{(1+\beta_1)(1+\beta_2)}Q^2\left\langle Q^2f_{\beta_2},f_{\beta_2}\right\rangle<0
		$$
		where in the last inequality we have used the fact that $\beta$ is strictly positive and {that both $Q$ and $f_{\beta_2}$ do not vanish.}
		Combining these two sides, we get $0<0$, which is impossible. This shows that $\lambda_{\beta_1}<\lambda_{\beta_2}$.
	\end{proof}	
	
	For the non-negative eigenvalues, we indeed have
	\begin{proposition}\label{prop:2}
		For each $\beta\in(0,1)$, the operator $L_\beta$, acting on $L^2_{rad}(\mathbb{R}^3)$, has no eigenvalue in $[0,1]$.
	\end{proposition}
	\begin{proof}
		{Let $\lambda \in[0,1]$.} Assume there exists a radial $H^2$ function $\psi$ satisfying $L_\beta\psi=\lambda \psi$. We know from the elliptic regularity theory (see for instance \cite{Ni1993}) that $\psi$ is smooth and decays exponentially at spatial infinity.
		Based on this regularity result, we are going to prove the following result, which embodies our key argument.
	\begin{claim}\label{lem:01:1}
		 $\psi$ is identically zero.
	\end{claim}
	\begin{proof}
		Recall $\lambda\in[0,1]$ and that $\psi$ satisfies $L_\beta\psi=\lambda\psi$. We argue by contradiction, assume that is not a zero function. Rewrite the equation
		\begin{equation}\label{eq:09301}
			\psi''+\frac{2}{r}\psi'-(1-\lambda)\psi+\frac{3-\beta}{1+\beta}Q^2\psi=0,\ r>0.
		\end{equation}	
		
		Denoting $\epsilon=1-\lambda\in [0,1]$ and $F_\epsilon := r\psi(r)$, we see that $F_\epsilon$ satisfies
		\begin{equation}
			F_\epsilon '' =\left(\epsilon-\frac{3-\beta}{1+\beta}Q^2\right) F_\epsilon
		\end{equation}
		Observe that $F_\epsilon(0)=0$. If the value of $F'_{\epsilon}(0)$ was zero, then it follows that $F_\epsilon$ is identically zero. Thus $F'_\epsilon$ does not vanish at $r=0$. Dividing $F_\epsilon$ by $-F'_\epsilon(0)$ and denoting the resulted function still by $F_\epsilon$, we are lead to consider
		\begin{equation}\label{eq:09302}
			\left\{
			\begin{split}
				F_\epsilon '' =\left(\epsilon-\frac{3-\beta}{1+\beta}Q^2\right) F_\epsilon\\
				F_\epsilon(0)=0, F_\epsilon'(0)=-1
			\end{split}
			\right. .
		\end{equation}
		
		\begin{claim}\label{claim:2}
			$F_\epsilon$ must change sign on $(0,\infty)$. In addition, by denoting by $r_\epsilon>0$ the first positive zero of $F_\epsilon$, we have $r_\epsilon\geq r_0$.	
		\end{claim}
		\begin{proof}
			Suppose otherwise $-F_\epsilon(r)$ remains positive for all $r>0$.
			{Note that $L_{\beta}$ has an unique negative eigenvalue $\lambda_\beta<0$ such that $L_\beta h=\lambda_\beta h$ for some strictly positive function $h\in L^2({\mathbb{R}^3})$}. Denoting $H:=rh$ and doing proper renormalization, we have the equations for $-F_\epsilon$ and $H$ respectively
			\begin{equation*}
				\left\{
				\begin{split}
					(-F_\epsilon)'' =\left(\epsilon-\frac{3-\beta}{1+\beta}Q^2\right)(- F_\epsilon)\\
					(-F_\epsilon)(0)=0,\;\;\;( -F_\epsilon)'(0)=1
				\end{split}
				\right. ,\;\;\;		
				\left\{
				\begin{split}
					H'' =\left(1-\lambda_\beta-\frac{3-\beta}{1+\beta}Q^2\right)H\\
					H(0)=0,\;\;\; H'(0)=1
				\end{split}
				\right. .
			\end{equation*}
			Since $\epsilon-\frac{3-\beta}{1+\beta}Q^2<1-\lambda_\beta-\frac{3-\beta}{1+\beta}Q^2$, it follows from Sturm's comparison theorem (see \cite{LY}) that $-F_{\epsilon}(r)\leq H(r)$ for all $r\geq 0$. On the other hand, we have the asymptotic behavior of $H\sim e^{-\sqrt{1-\lambda_\beta}r}$ and $-F_{\epsilon}\sim e^{-\sqrt{\epsilon}r}$ for all $r\gg 1$. Combining these two points, we get $\epsilon\geq 1-\lambda_\beta$, which contradicts with the assumption $\epsilon=1-\lambda<1-\lambda_\beta$. Thus $F_\epsilon$ changes sign.
			
			The second part of the asserted results, follows from another application of Sturm's comparison theorem. This completes the proof of Claim \ref{claim:2}.
		\end{proof}
		
		Next, we compare $F_\epsilon$ with the unique function $G$, which obeys
		\begin{equation}\label{eq:09303}
			\left\{
			\begin{split}
				G '' &=-3Q^2 G \\
				G(0)&=0, G'(0)=-1
			\end{split}
			\right. .
		\end{equation}
		It is shown in \cite{LY} that $G$ changes sign exactly once in $(0,+\infty)$. Here we denote the corresponding unique positive zero of  $G$ by $r_*$. Since $-\frac{3-\beta}{1+\beta}Q^2\geq -3Q^2 $, we infer from Sturm's comparison theorem that $r_0\geq r_*>0$. As a consequence, we have as well $r_\epsilon\geq r_0$.	
		
		Finally, the functions $F(r):=F_\epsilon(r+r_\epsilon)/F'_\epsilon(r_\epsilon)$ and $G_\ast(r):=G(r+r_\ast)/G'(r_\ast)$ satisfy respectively the following equations
		\begin{equation}
			\left\{
			\begin{split}
				F '' &=\left(\epsilon-\frac{3-\beta}{1+\beta}Q^2(r+r_\epsilon)\right) F\\
				F(0)&=0, F'(0)=1
			\end{split}
			\right. ,\;\;\;
			\left\{
			\begin{split}
				G_* '' &=(-3Q^2(r+r_\ast)) G_* \\
				G_*(0)&=0, G_*'(0)=1
			\end{split}
			\right. .	
		\end{equation}
		We now compare these two shifted functions.	Since $r_\ast\leq r_\epsilon$ and $Q$ is decreasing, $\epsilon-\frac{3-\beta}{1+\beta}Q^2(r+r_\epsilon)\geq -3Q^2(r+r_\ast)$. It follows from Sturm's comparison theorem that $F\geq G_*$. However, we recall Theorem 6.1 in \cite{LY} that $G_*(r)$ stays positive for $r>0$ and $\min_{r\geq 1} \frac{G_*(r)}{r}\geq c_0$ for some $c_0>0$. Combining these two points, we see that $F_\epsilon(r)\geq c_0r$ for all sufficiently large $r$, which implies that $F_\epsilon \notin L^2((0,+\infty),dr)$. This exclusion in turn implies the original function $\psi$ is not in $L^2(\mathbb{R}^3)$, contradicting the initial assumption. This finishes the proof of Claim \ref{lem:01:1}.
	\end{proof}	
		It follows from Claim \ref{lem:01:1} that $\psi\equiv0$, which implies there is no eigenvalue in $[0,1]$ for $L_\beta$, acting on $L^2_{rad}(\mathbb{R}^3)$. This completes the proof of Proposition \ref{prop:2}.
	\end{proof}
	
	Proposition \ref{prop:2} implies the item $(iii)$ in Theorem \ref{thm:main}. {Since} $L^2_{rad}(\mathbb{R}^3)$ is a strict subspace of $L^2(\mathbb{R}^3)$, the asserted result $(ii)$ in Theorem \ref{thm:main} can not be implied by this proposition. Thus we shall give
	
	\begin{proof}[Proof of (ii)]	
		Let $\psi\in H^2(\mathbb{R}^3)$ be a solution to
		\begin{equation}
			L_\beta\psi=0.
		\end{equation}
		We shall show
		{$\psi$ is a zero function}. We argue by contradiction, assuming that $\psi$ is a nonzero function. We will achieve the contradiction by analyzing each coefficient in the spherical expansion of $\psi$. For this, we first recall facts about eigenvalues and the corresponding eigenfunctions of the standard Laplacian on the two dimensional sphere $\mathbb{S}^2$.
		
		Let $\mu_k,e_k(\omega)$ be the eigen-datum for $-\Delta_{\mathbb{S}^2}$. Then we have
		\begin{itemize}
			\item $\mu_0=0<\mu_1=\mu_2=\mu_3=2<\mu_4\leq \cdots$
			\item $\{e_k\}_{k\geq0}$ is an orthonormal basis for $L^2(\mathbb{S}^2)$.
		\end{itemize}
		For each $k\geq 0$, we set
		\begin{equation}
			\psi_k(r):=\int_{\mathbb{S}^2}\psi(r,\omega)e_k(\omega)d\omega.
		\end{equation}
		Then $\psi_k(r)$ decays exponentially to zero as $r$ tends to infinity and it obeys the equation
		\begin{equation}\label{eq:04170}
			\psi_k''+\frac{2}{r}\psi_k'-\psi_k+\left(\frac{3-\beta}{1+\beta}Q^2-\frac{\mu_k}{r^2}\right)\psi_k=0,\ r>0.
		\end{equation}
	
		Note that for $k=0$, the equation \eqref{eq:04170} is just the equation \eqref{eq:09301} with $\lambda=0$. Thus Claim \ref{lem:01:1} implies that $\psi_0\equiv0$ on $(0,\infty)$. Therefore, we shall show that $\psi_k\equiv0$ for each $k\geq 1$.
		
	    The radial part of $\psi$ has been just proved to be identically zero. Thus $\psi_k$ is not identically equaling $0$ for some $k\geq 1$. By multiplying $\psi_k$ by $-1$ if necessary, we may assume $\psi_k(0)\geq 0$. {By a simple contradiction argument (the same as that at the beginning of the proof of Claim \ref{lem:01:1}), we can obtain $\psi_k(0)>0$}. Therefore, there exists a $\rho_k\in(0,\infty]$ so that
		$\psi_k(r)>0$ for $0<r<\rho_k$ and $\psi_k(\rho_k)=0$.
		
		Multiplying \eqref{eq:04170} by $Q'(r)r^2$, integrating the resulted equality over the interval $(0,\rho_k)$ and applying the integration by parts, we get
		\begin{equation}\label{eq:04171}
			\rho_k^2\psi'_k(\rho_k)Q'(\rho_k)+\int_0^{\rho_k}\left(Q'''+\frac{2}{r}Q''-Q'+{\frac{3-\beta}{1+\beta}}Q^2Q'\right)\psi_kr^2dr-\mu_k\int_0^{\rho_k}Q'\psi_kdr=0.
		\end{equation}
		Since $Q$ is the fundamental solution, it also satisfies the equation
		\begin{equation}
			Q'''+\frac{2}{r}Q''-\frac{2}{r^2}Q'-Q'+3Q^2Q'=0
		\end{equation}
		Inserting this into \eqref{eq:04171} and using the relationship between $L_0$ and $L_\beta$, we obtain
		\begin{equation}\label{eq:04172}
			\rho_k^2\psi'_k(\rho_k)Q'(\rho_k) +\left(2-\mu_k\right)\int_0^{\rho_k}Q'\psi_kdr-\frac{4\beta}{1+\beta}\int_0^{\rho_k} Q^2Q'\psi_k r^2dr=0.
		\end{equation}
		
		Noting that $Q'(r)<0$ for each $r>0$. At point $r=\rho_k$, {it is easy to see} that $\psi'_k(\rho_k)\leq 0$. Thus the first term on the left hand side of \eqref{eq:04172} is nonnegative. Together with the assumption $\beta >0$, we see the third term is strictly positive. It follows from $k\geq 1$ that $\mu_k\geq 2$ and hence $2-\mu_k\leq 0$, so that the second term is positive. Combining all these together, we see that the left hand side of \eqref{eq:04172} is strictly positive while the right hand side is zero. This gives the contradiction, since \eqref{eq:04172} is an identity. This completes the proof.
	
	\end{proof}
	
	Note that Proposition \ref{prop:2} also implies that the point `1' is not an eigenvalue of $L_\beta$ on $L^2_{rad}(\mathbb{R}^3)$, which is just the first half of $(iv)$. Therefore, to finish the proof, it suffices to prove
	\begin{claim}
		{The point `1' is not a resonance of $L_\beta$ over $L^2_{rad}(\mathbb{R}^3)$.}	
	\end{claim}
	\begin{proof}
		{Indeed, the proof of Claim \ref{lem:01:1} implies as well that the {threshold point} $1$ is not a resonance of $L_\beta$.} With the notations as in the proof of Claim \ref{lem:01:1}, we have
		\begin{equation}
			F_0(r)\geq c_0r,\ \ \forall r\geq 1,
		\end{equation}
		where $c_0$ is a strictly positive number. Using the change of variable, we can compute for any $\gamma \leq 3$
		\begin{equation}
			\int_{\mathbb{R}^3}|\psi(x)|^2(1+|x|)^{-\gamma}dx\geq c_1\int_{1}^\infty|F_0(r)|^2(1+r)^{-\gamma}dr\geq c_1c^2_0\int_1^\infty(1+r)^{-\gamma}r^2dr=\infty.
		\end{equation}	
		Here $c_1$ is another positive constant. By taking $\gamma=0$, we see that $\psi\notin L^2(\mathbb{R}^3)$. Combining this point with the above inequality, we can conclude from {the definition of resonance} that the threshold point $1$ is not a resonance.
		
	\end{proof}

%
%
%

 
\end{document}